\documentclass[a4paper,12pt,leqno]{amsart}
\usepackage{amsmath,amsfonts,amsthm,amssymb,dsfont}
\usepackage[alphabetic]{amsrefs}
\usepackage[OT4]{fontenc}
\usepackage{enumerate}
\usepackage{mathrsfs, mathtools}
\usepackage{enumerate, xspace}
\usepackage[pdftex]{graphicx}
\usepackage{color}

\long\def\symbolfootnote[#1]#2{\begingroup
\def\thefootnote{\fnsymbol{footnote}}\footnote[#1]{#2}\endgroup}

\newtheorem{theorem}{Theorem}[section]
\newtheorem{lemma}[theorem]{Lemma}
\newtheorem{thm}[theorem]{Theorem}

\newtheorem{cor}[theorem]{Corollary}

\theoremstyle{definition}
\newtheorem{rem}[theorem]{Remark}
\newtheorem{defin}[theorem]{Definition}

\newcommand{\N}{\mathbb{N}}
\newcommand{\Z}{\mathbb{Z}}
\newcommand{\Mod}{\mathrm{Mod}(S)}
\begin{document}

\title[Unicorn paths and hyperfiniteness]{Unicorn paths and hyperfiniteness for the mapping class group}

\author[P.~Przytycki]{Piotr Przytycki$^{\dag}$}
\address{Dep.\ of Math.\ \& Stat., McGill University\\
Burnside Hall, 805 Sherbrooke St.\ W\\
Montreal, Quebec, Canada H3A 0B9}
\email{piotr.przytycki@mcgill.ca}
\thanks{$\dag \ddag$ Partially supported by NSERC and National Science Centre, Poland UMO-2018/30/M/ST1/00668.}

\author[M.~Sabok]{Marcin Sabok$^{\ddag}$}
\email{marcin.sabok@mcgill.ca}
\thanks{$\dag$ Partially supported by AMS}

\maketitle

\begin{abstract}
\noindent
Let $S$ be an orientable surface of finite type. Using Pho-On's infinite unicorn paths, we prove the hyperfiniteness of orbit equivalence relations induced by the actions of the mapping class group of~$S$ on the Gromov boundaries
of the arc graph and the curve graph of~$S$. In the curve graph case, this strengthens the results of Hamenst\"adt and Kida that this action is universally amenable and that the mapping class group of~$S$ is exact.
\end{abstract}

\section{Introduction}
\label{sec:intro}
An equivalence relation $E$ on a standard Borel space $X$ is
\textit{Borel} if $E$ is a Borel subset of $X\times X$. An
equivalence relation is \textit{countable}
(resp.\ \textit{finite}) if every equivalence class is countable (resp.\ finite). Given a Borel action of a
countable group on a standard Borel space $X$, the
induced orbit equivalence relation is a countable Borel
equivalence relation. A Borel equivalence relation~$E$ is \textit{hyperfinite} if
$E$ can be written as an increasing union of a sequence of finite Borel
equivalence relations.

Let $S$ be an oriented surface of genus $g\geq 0$ with $n\geq 0$ punctures, of negative Euler characteristic. We denote by~$\mathcal A(S)$ (for $n\geq 1$) and~$\mathcal {C}(S)$ its arc graph and its curve graph, which are Gromov hyperbolic (see Section~\ref{sec:paths}). The actions of the mapping class group $\Mod$ on~$\mathcal A(S)$ and~$\mathcal {C}(S)$ by automorphisms extend to actions on their Gromov boundaries by homeomorphisms. Our main result is:

\begin{thm}
\label{thm:arc} The orbit equivalence relation on $\partial \mathcal A(S)$ induced by the action of $\Mod$ is hyperfinite.
\end{thm}

As a consequence we will derive:

\begin{cor}
\label{cor:curve}
The orbit equivalence relation on $\partial \mathcal {C}(S)$ induced by the action of $\Mod$ is hyperfinite.
\end{cor}

This strengthens the results of Hamenst\"adt \cite[Cor~2]{Ham} and Kida \cite[Thm~1.4(ii)]{K}
that this equivalence relation is universally amenable (for definition, see Section~\ref{sec:universal}). 

We will also obtain the following (for the definition of $\mathcal {CL}(S)$, see Section~\ref{sec:complete}):

\begin{cor}
\label{cor:CL}
The orbit equivalence relation on the space of complete geodesic laminations $\mathcal {CL}(S)$ induced by the action of $\Mod$ is hyperfinite.
\end{cor} 
In particular, this equivalence relation is universally amenable. 
Since $\mathcal {CL}(S)$ is compact and Hausdorff, and its point stabilisers are virtually abelian, this gives a new proof of \cite[Thm~1]{Ham} that the action of $\Mod$ on  $\mathcal {CL}(S)$ is topologically amenable (see Section~\ref{sec:universal}). 
This implies that $\Mod$ is exact (which was proved independently to Hamenst\"adt by Kida \cite[Thm~C.5]{K}).

\subsection{Amenability}
\label{sec:universal}
Let $\mu$ be a Borel probability measure on a standard Borel space $X$.
The notion of amenability for a measurable action of a countable group on $(X,\mu)$ was introduced by Zimmer \cite{Z} (see also \cite{zim}) and has many equivalent definitions (see e.g.\ \cite{aeg}; see \cite{Renault} for a more general definition for measured groupoids). It is closely related to the notion of amenability for countable Borel equivalence relations \cite[Def~7.4.5]{gao} (see also \cite[\S3]{Kech1} or \cite[\S3]{Moore}). Namely, a measurable action on $(X,\mu)$ is $\mu$-amenable if and only if $\mu$-almost all stabilisers are amenable and the induced orbit equivalence relation on~$X$ is $\mu$-amenable \cite[Thm~5.1]{aeg}. A countable Borel equivalence relation on~$X$ or a Borel action of a countable group on~$X$ is \emph{universally amenable} if it is $\mu$-amenable for every quasi-invariant Borel probability measure $\mu$ on $X$ (or, equivalently, for every Borel probability measure $\mu$ on $X$ \cite[Cor~10.2]{KM}). In particular, any Borel action of an amenable group is universally amenable.
 
By \cite[Thm~3.3.7]{Renault}, an action of a countable group $G$ by homeomorphisms on a locally compact Hausdorff space $X$ is universally amenable if and only if it is topologically amenable (for definition, see \cite[Def~2.1]{Oza}). If $X$ is compact Hausdorff, then the topological amenability of the action implies the exactness of $G$ \ \cite[Thm~7.2]{AD2}.
 
Boundary actions have been studied extensively from the
point of view of amenability. Connes, Feldman, and Weiss \cite[Cor~13]{cfw} and, independently, Vershik \cite[Thm~2]{ver},
showed that the tail equivalence relation discussed in Section~\ref{subsec:hyper} is universally amenable, which implies that the induced action of the finitely generated free group~$F_n$ on its Gromov boundary $\partial F_n$
is universally amenable.
This was later generalised by
Adams \cite{ada} to all hyperbolic groups (see also \cite{Kai}).
Furthermore, Ozawa \cite{O} proved that the action of a relatively hyperbolic group with amenable parabolic subgroups on the Gromov boundary of its coned-off Cayley graph is topologically amenable. Moreover, Nevo and Sageev proved that the action of a cocompactly cubulated group on a particular subset of its Roller boundary is universally amenable \cite{NS}.
L\'{e}cureux proved that if a group $G$ acts geometrically on a building~$X$, then the action of $G$ on the combinatorial boundary of $X$ is topologically amenable \cite{L}. Finally, Bestvina, Guirardel, and Horbez proved that the action of $\mathrm{Out}(F_n)$ on the Gromov boundary of its free factor complex is universally amenable, see \cite[Thm~6.4]{BGH} and \cite[Prop~7.2]{GHL} (which uses the description of the Gromov boundary of the free factor complex in \cite{BR} and \cite{H2}).

\subsection{Hyperfiniteness}
\label{subsec:hyper}
As shown independently
by Weiss and Slaman--Steel \cite[Thm~7.2.4]{gao}, a Borel equivalence relation $E$ is
hyperfinite if and only if there is a Borel action of
$\mathbb{Z}$ inducing $E$ as its orbit equivalence
relation. Since $\Z$ is amenable, the hyperfiniteness of a Borel equivalence relation implies its universal amenability.
It is a well-known open problem, whether the converse holds, i.e.\ whether a universally amenable Borel equivalence relation is always hyperfinite. Connes, Feldman, and Weiss \cite[Thm~10]{cfw} (see also \cite[Thm~10.1]{KM}) showed that a $\mu$-amenable Borel equivalence relation on $X$ becomes hyperfinite after removing from $X$ a set of $\mu$-measure $0$.

The relative complexity of Borel equivalence relations is
measured by Borel reducibility. Given two equivalence
relations $E$ and $F$ on standard Borel spaces $X$ and $Y$,
respectively, a function $f:X\to Y$ is a \textit{Borel reduction} from $E$ to $F$ if $f$ is a Borel function and
for every $a,b\in X$ we have $a\sim_E b$ if and
only if $f(a)\sim_F f(b)$. A relation $E$ is
\textit{Borel reducible} to $F$, if
there exists a Borel reduction from $E$ to $F$.
The relation $E_0$ is defined
on $\{0,1\}^\mathbb{N}$ (with the product topology) as $(a_i)_{i=0}^\infty\sim_{E_0}(b_i)_{i=0}^\infty$ if $a_i=b_i$ for all~$i$ sufficiently large. It is easy
to see that $E_0$ is hyperfinite. In fact, a
countable Borel equivalence relation is hyperfinite if and only if it is
Borel reducible to~$E_0$ \cite[Thm~7.2.2]{gao}.

Let $\Omega$ be a countable set with discrete topology. The \textit{tail equivalence relation} $E_t$
on $\Omega^\mathbb{N}$ is defined as $(a_i)_{i=0}^\infty\sim_{E_t}(b_i)_{i=0}^\infty$ if
there exists $k\in\mathbb{Z}$ such that $a_i=b_{i+k}$ for
all~$i$ sufficiently large. Dougherty, Jackson, and Kechris showed that
$E_t$ is Borel reducible to $E_0$, and so it is hyperfinite \cite[Cor~8.2]{djk}.
It is not hard to see that the orbit equivalence relation induced by the action of $F_n$ on $\partial F_n$ is Borel reducible to $E_t$ with finite $\Omega$. Hence that orbit equivalence relation on $\partial F_n$ is hyperfinite, which we will shortly express by saying that the \emph{boundary action of $F_n$ is hyperfinite}.

More recently, Huang, Sabok, and Shinko \cite{hss} showed that for cocompactly cubulated
hyperbolic groups, their boundary actions are
hyperfinite. The proof relied on a study of geodesic ray
bundles in hyperbolic groups. While Touikan \cite{T} showed that that
approach does not work for arbitrary hyperbolic groups,
Marquis \cite{M} used it to prove the hyperfiniteness of boundary actions
of groups acting cocompactly on locally finite hyperbolic buildings with trivial chamber stabilisers.
Very recently, Marquis and Sabok \cite{ms} showed the
hyperfiniteness of the boundary action of an
arbitrary hyperbolic group.

\smallskip

\textbf{Organisation.} In Section~\ref{sec:paths} we recall the basics on arcs, laminations, and unicorn paths. In Section~\ref{sec:key} we prove a pair of key lemmas: the local characterisation of Pho-On's infinite unicorn paths, and the tail equivalence for asymptotic infinite unicorn paths. This allows for the proofs of Theorem~\ref{thm:arc} and Corollary~\ref{cor:curve} in Section~\ref{sec:hyper}. We prove Corollary~\ref{cor:CL} in Section~\ref{sec:complete}.

\smallskip

\textbf{Acknowledgements.} We thank Camille Horbez for the input on the $\mathrm{Out}(F_n)$ case, Jean Renault for helpful explanations, and the Referee for valuable remarks.

We thank Antoni Sabok-Przytycki for encouragement.

\section{Unicorn paths}
\label{sec:paths}
\subsection{Arcs and laminations}
\label{subsec:Arcs}

As in the introduction, $S$ is obtained from a closed oriented surface of genus $g$ by removing $n$ points. Thus $S$ has $n$ topological ends, which are called \emph{punctures}. An \emph{oriented arc} on $S$ is a map from $(0,1)$ to $S$ that is proper. A proper map induces a map between topological ends of spaces, and in this sense each endpoint of $(0,1)$ is sent to a puncture of $S$. We will say that the oriented arc \emph{starts} and \emph{ends} at these punctures. A \emph{homotopy} between oriented arcs $a$ and $b$ is a proper map $(0,1)\times[0,1]\to S$ whose restriction to $(0,1)\times \{0\}$ equals~$a$ and whose restriction to $(0,1)\times \{1\}$ equals $b$. In particular, $a$ and $b$ start at the same puncture and end at the same puncture. A \emph{curve} on $S$ is a map from a circle $S^1$ to~$S$.

An oriented arc or a curve is \emph{simple} if it is an embedding. In that case we can and will identify the oriented arc or the curve with its image in~$S$. We record, however, the orientation of the arc, while for the curve we discard it. A curve is \emph{essential} if it is not homotopically trivial. A curve $c\colon S^1\to S$ is \emph{non-peripheral} if it cannot be homotoped into the puncture in the sense that there is no proper map $S^1\times[0,1)\to S$ whose restriction to $S^1\times \{0\}$ is $c$. An oriented arc $a\colon (0,1)\to S$ is \emph{essential} if there is no proper map $(0,1)\times[0,1)\to S$ whose restriction to $(0,1)\times \{0\}$ is $a$. Unless otherwise stated, all oriented arcs in the article are simple and essential, and all curves are simple, essential and non-peripheral.

Suppose that the Euler characteristic $\chi=2-2g-n$ of $S$ is negative. If $n\geq 1$, the \emph{arc graph} $\mathcal {A}(S)$ is the graph whose vertex set $A$ is the set of homotopy classes of oriented arcs on $S$. Two vertices in $A$ are connected by an edge if they can be realised disjointly. Note that since our arcs are oriented, our $\mathcal {A}(S)$ differs from the usual arc graph by replacing each vertex by two.

Allow now $n=0$, but suppose that we are not in one of the exceptional cases where $g=0$ and $n=3$ or $4$, or $g=1$ and $n=1$. Then the \emph{curve graph} $\mathcal {C}(S)$ is the graph whose vertices are the homotopy classes of curves on $S$. Again, two vertices are connected by an edge if they can be realised disjointly. In the exceptional cases the edges of $\mathcal {C}(S)$ are defined differently, but we will not be appealing to that definition in our article. By~\cite{MM} and \cite{MS}, the graphs $\mathcal {C}(S)$ and $\mathcal{A}(S)$ are Gromov-hyperbolic.

We fix an arbitrary complete hyperbolic metric on~$S$.
A \emph{geodesic lamination} on~$S$ is a compact subset of $S$ that is a disjoint union of \emph{leaves} that are geodesic lines and circles in $S$ that do not self-intersect. A geodesic lamination $L$ is \emph{minimal} if its every leaf is dense in $L$. Let $Y\subseteq S$ be a subsurface whose all boundary components are geodesic circles. We say that a geodesic lamination $L\subset Y$ \emph{fills} $Y$ if every curve on $Y$ intersects~$L$. Analogously, a pair of oriented arcs $a,b\subset Y$ \emph{fills} $Y$ if every curve on $Y$ intersects the geodesic representative of $a$ or $b$.

A \emph{peripherally ending lamination} is a minimal geodesic lamination that fills a subsurface $Y$ containing all the punctures of $S$. An \emph{ending lamination} is a minimal geodesic lamination that fills the entire~$S$. Let $\mathcal{EL}(S)\subset \mathcal {EL}_0(S)$ denote the sets of ending, and peripherally ending laminations on $S$, respectively, with the topology given by the following \emph{coarse Hausdorff} convergence. Namely, $L_n\xrightarrow{\text{CH}}L$ if for any subsequence $L_{n_k}$ Hausdorff converging to a geodesic lamination $L'$, we have $L\subset L'$ (see \cite{H1}).
By \cite{Kla} and \cite{Sch} (see also Theorem~\ref{thm:Pho-On} in Section~\ref{sec:key}), the spaces $\mathcal{EL}(S), \mathcal {EL}_0(S)$ can be equivariantly identified with the Gromov boundaries of $\mathcal {C}(S)$ and $\mathcal {A}(S)$.

\subsection{Unicorns}
As in Section~\ref{subsec:Arcs}, let $A$ denote the set of homotopy classes of oriented arcs on $S$.

\begin{defin}
Let $a,b\in A$, and keep the notation $a,b$ for the geodesic oriented arcs representing them. A \emph{unicorn arc} for $a$ and $b$ is the homotopy class of an oriented arc that is a concatenation $a'\cup b'$ for $a'$ an initial segment of~$a$, and $b'$ a terminal segment of $b$, possibly $a'=a,b'=\emptyset,$ or $a'=\emptyset,b'=b$. Note that orienting the arcs replaces the choice of endpoints in \cite[Def~3.1]{HPW}.

The set of all oriented arcs that are such concatenations $a'\cup b'$ can be ordered into a sequence $(a'_i\cup b'_i)_{i=0}^n$ so that for all $0\leq i<n$ we have $a'_{i+1}\subset a'_i$ and $b'_{i+1}\supset b'_i$. We denote by $c_i\in A$ the homotopy class of $a'_i\cup b'_i$ and
we call the sequence $P(a,b)=(c_i)_{i=0}^n\in A^{n+1}$ the \emph{unicorn path} from $a$ to $b$.
\end{defin}

Note that we have $c_0=a$ and $c_n=b$. Moreover, the unicorn path is indeed an edge-path in $\mathcal{A}(S)$:

\begin{rem}[{\cite[Rm~3.2]{HPW}}]
\label{rem:adjacent}
For each $0\leq i<n$, the unicorn arcs $c_i,c_{i+1}$ are adjacent in $\mathcal{A}(S)$.
\end{rem}

Let $L_0$ be a peripherally ending lamination. Let $l$ be a geodesic line on $S$ that does not self-intersect and ends at a puncture in the sense that $l$ contains a geodesic ray properly embedded in $S$. We say that $l$ is \emph{asymptotic to $L_0$}, if $l\subset S\setminus L_0$. Since each puncture of $S$ lies in a once-punctured ideal polygon of $S\setminus L_0$, the number of such $l$ is bounded by the total number of their ideal vertices, which is at most $2|\chi|$.

\begin{defin}[{\cite[\S3.1]{Pho}}]
Let $a\in A$ and keep the notation $a$ for the geodesic oriented arc representing it. Let $l$ be a geodesic line asymptotic to $L_0\in \mathcal {EL}_0(S)$. A \emph{unicorn arc} for $a$ and $l$ is the homotopy class an oriented arc that is a concatenation $a'\cup l'$ for $a'$ an initial segment of $a$, and $l'$ a terminal segment of $l$, possibly $a'=a$ and $l'=\emptyset$.

The set of all oriented arcs that are such concatenations $a'\cup l'$ can be ordered into a sequence $(a'_i\cup l'_i)_{i=0}^\infty$ so that for all $i\geq 0$ we have $a'_{i+1}\subset a'_i$ and $l'_{i+1}\supset l'_i$. We denote by $c_i\in A$ the homotopy class of $a'_i\cup l'_i$ and
we call the sequence $P(a,l)=(c_i)_{i=0}^\infty\in A^\N$ the \emph{infinite unicorn path} from $a$ to $l$.
\end{defin}

\section{Key lemmas}
\label{sec:key}

\begin{defin}
Let $n\in \{3,4,\ldots, \infty\}$. A sequence $(c_i)_{i=0}^n\in A^{n+1}$ is a \emph{locally unicorn path} if for each $0\leq j<k\leq n$ with $j+3\leq k<\infty$, the sequence $(c_i)_{i=j}^k$ is the unicorn path from $c_j$ to $c_k$.
\end{defin}

By Remark~\ref{rem:adjacent}, a locally unicorn path is an edge-path in $\mathcal{A}(S)$. Moreover, by \cite[Lem~3.5]{HPW} each finite unicorn path of length $\geq 3$ is a locally unicorn path. Furthermore, by \cite[Lem~3.4]{Pho} an infinite unicorn path is also locally unicorn.

By \cite[Prop~4.2]{HPW} there is a universal constant $C$ such that each finite unicorn path $P(a,b)$ is at Hausdorff distance $\leq C$ from a geodesic edge-path in $\mathcal{A}(S)$ from $a$ to $b$. Consequently, each locally unicorn path is bounded or converges w.r.t.\ the Gromov product (see \cite[\S7.2]{Ghys}) to a point in $\partial \mathcal{A}(S)$. This leads to the following result of Pho-On (the existence of an equivariant homeomorphism was announced earlier by Schleimer \cite{Sch}).

\begin{thm}[{\cite[\S3.2-3]{Pho}}]
\label{thm:Pho-On}
Let $a\in A$. Let $L_0\in \mathcal {EL}_0(S)$ and let $l$ be a geodesic line asymptotic to~$L_0$. Then $P(a,l)$ is not bounded and its limit $F(L_0)\in \partial \mathcal{A}(S)$ w.r.t.\ the Gromov product depends only on $L_0$. Furthermore,
$F\colon \mathcal{EL}_0(S)\to \partial \mathcal A(S)$ is a $\Mod$-equivariant homeomorphism.
\end{thm}

In fact, the local condition characterises infinite unicorn paths:

\begin{lemma}
\label{lem:local characterization}
Let $P$ be a locally unicorn path that is not bounded in~$\mathcal{A}(S)$.
Then~$P$ is an infinite unicorn path.
\end{lemma}

\begin{proof}
Denote $P=(c_i)_{i=0}^\infty\in A^\N$, and keep the notation $c_i$ for the geodesic oriented arcs representing them. Since $P$ is not bounded in~$\mathcal{A}(S)$, it converges to some point $F(L_0)\in\partial \mathcal{A}(S)$. By \cite[Lem~3.9]{Pho}, we have that $c_i$ coarse Hausdorff converge to $L_0\in \mathcal {EL}_0(S)$. Denote $c=c_0$. We claim that for each $n\geq 1$ there is a geodesic line $l$ asymptotic to~$L_0$ such that for each $i\leq n$ the unicorn arc $a_i$ on the infinite unicorn path from $c$ to $l$ coincides with $c_i$. Since there are only finitely many $l$ asymptotic to $L_0$, the lemma follows from the claim.

\begin{figure}
\begin{center}
\includegraphics[width=0.7\textwidth]{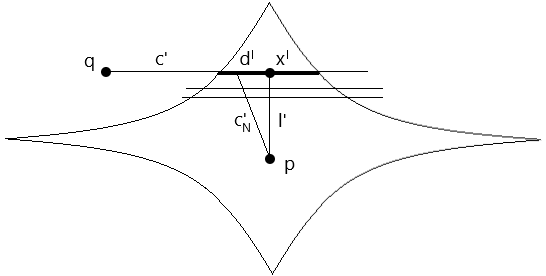}
\end{center}
\caption{Ideal polygon $D$}
\label{fig1}
\end{figure}

To justify the claim, note that since $P$ is a locally unicorn path, all $c_i$ with $i\geq 1$ end at a common puncture $p$. Let $D$ be the ideal polygon of $S\setminus L_0$ containing~$p$. Let $l$ be a geodesic line asymptotic to~$L_0$ ending at $p$. Let $c'\cup l'$ represent the $n$-th unicorn arc on the unicorn path from $c$ to $l$, let $x^l=c'\cap l'$, and let $d^l$ be the segment of~$c$ that is the component of $D\cap c$ containing $x^l$. See Figure~\ref{fig1}. Let $D_p$ be the component of $D-\bigcup_ld^l$ containing $p$, where the union is taken over all the geodesic lines $l$ asymptotic to $L_0$ and ending at $p$. Let $\alpha>0$ be the minimum possible angle that makes with $L_0$ a geodesic ray in $D_p$ starting on $L_0$ and ending at $p$. Since $(c_i)_{i=0}^\infty\xrightarrow{\text{CH}}L_0$, there is $N\geq n$ such that $c_N$ does not intersect $L_0$ at angle $\geq \alpha$. Consequently, the component $c_N'$ of $c_N\cap D_p$ ending at $p$ starts on $d^l$ for some $l$ (see Figure~\ref{fig1}).

Let $q$ be the puncture at which $c$ starts. We have a bijection $h\colon l'\cap c\to c_N'\cap c$ such that each pair $x,h(x)$ lies in the same component of $D_p\cap c$. Furthermore, for each $x\in l'\cap c$, the segments $qx\subset c$ and $xp\subset l$ intersect only at $x$ if and only if the segments $qh(x)\subset c$ and $h(x)p\subset c_N$ intersect only at $h(x)$. In other words, the concatenation $qx\cup xp$ represents a unicorn arc for $c$ and $l$ if and only if the concatenation $qh(x)\cup h(x)p$ represents a unicorn arc for $c$ and $c_N$.
Moreover, these two oriented arcs are homotopic. Finally, this correspondence preserves the order of unicorn arcs. Thus, for $0\leq i\leq n,$ we have $a_i=c_i$, justifying the claim.
\end{proof}

Let $L_0\in \mathcal{EL}_0(S)$. We define an equivalence relation $\sim_{L_0}$ on $A$, by declaring $a\sim_{L_0}b$ if the geodesic representatives of $a,b$ start at the same puncture and their first points in~$L_0$ lie on the same side of the ideal polygon of $S\setminus L_0$ containing that puncture. Note that $\sim_{L_0}$ has at most $2|\chi|$ equivalence classes.

We now prove a tail equivalence lemma that will later allow us to reduce the orbit equivalence on $\partial \mathcal A(S)$ to $E_t$.

\begin{lemma}
\label{lem:unite}
Let $L_0\in \mathcal{EL}_0(S)$, and let $a,b\in A$ with $a\sim_{L_0}b$. Then for each geodesic line $l$ asymptotic to $L_0$, the unicorn path $(a_i)_{i=0}^\infty$ from $a$ to $l$ and the unicorn path $(b_i)_{i=0}^\infty$ from $b$ to $l$ satisfy $a_i=b_{i+k}$ for some $k\in \Z$ and all $i$ sufficiently large.
\end{lemma}
\begin{proof}
Let $a,b\in A$ with $a\sim_{L_0}b$ and keep the notation $a,b$ for the geodesic oriented arcs representing them. Let $x,y$ be the first points on $a,b$ in $L_0$. Since $a\sim_{L_0}b$, there is a geodesic segment $xy\subset L_0$. Furthermore, since $L_0$ is minimal, there are segments $xx',yy'$ in $a,b$ such that $x'y'$ is a geodesic segment in $L_0$, and $xx'y'y$ bounds a topological disc $B$ embedded in~$S$. See Figure~\ref{fig2}.

Consequently, the components of the intersection $B\cap l$ are geodesic segments joining $xx'$ to $yy'$, which yields a bijection $h\colon xx'\cap l\to yy'\cap l$. Let $a'\subset a$ be an initial segment of $a$ ending in $z\in xx'\cap l$, and let $b'$ be the initial segment of~$b$ ending in $h(z)$. Furthermore, let $l',l''$ be the terminal segments of $l$ starting in $z,h(z)$, respectively. Assume without loss of generality $l''\subset l'$, as in Figure~\ref{fig2}.

\begin{figure}
\begin{center}
\includegraphics[width=0.45\textwidth]{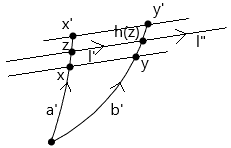}
\end{center}
\caption{Rectangle $B$}
\label{fig2}
\end{figure}

Note that $xz$ intersects $zh(z)\subset l'$ only at $z$. Furthermore, $xz$ is disjoint from $l''$ if and only if $yh(z)$ intersects $l''$ only at $h(z)$. Consequently, the concatenation $a'\cup l'$ represents a unicorn arc if and only if $b'\cup l''$ represents a unicorn arc. Moreover, these two oriented arcs are homotopic. Finally, this correspondence preserves the order of unicorn arcs, and all $a_i$ for $i$ sufficiently large are accounted for in this way.
\end{proof}

\section{Hyperfiniteness}
\label{sec:hyper}
We fix a basepoint $a_0\in A$.

\begin{defin}
Let $\mathcal P\subset A^\N$ be the set of infinite unicorn paths from~$a_0$ to any geodesic line asymptotic to any $L_0\in \mathcal {EL}_0(S)$. Let $f\colon \mathcal P\to \partial \mathcal A(S)$ be the map assigning to such path its limit $F(L_0)\in \partial \mathcal A(S)$ w.r.t.\ the Gromov product (see Theorem~\ref{thm:Pho-On}).
\end{defin}

Note that $f$ is finite-to-one, since $a_0$ is fixed and there are finitely many geodesic lines asymptotic to a given $L_0\in \mathcal {EL}_0(S)$.

\begin{rem}
\label{rem:Borel}
We equip the countable set $A$ with the discrete topology and $A^\N$ with the product topology.
Then the set $\mathcal P\subset A^\N$ is Borel. Indeed, by Lemma~\ref{lem:local characterization}, $\mathcal P$ is the set of locally unicorn paths that are not bounded. The set of locally unicorn paths is closed in~$A^\N$, since each of the conditions on $(c_i)_{i=j}^k$ to be a unicorn path is closed. Furthermore, for each $n\geq 0$, the set of sequences in $A^\N$ at distance $\leq n$ from $a_0$ is closed, so the set of sequences in $A^\N$ at bounded distance from~$a_0$ is a countable union of closed sets. Consequently, $\mathcal P$ is a countable intersection of open sets.

Since locally unicorn paths are uniformly Hausdorff close to geodesic edge-paths, and the function $f$ assigns their limits in $\partial \mathcal A(S)$, we have that $f$ is continuous w.r.t.\ the metric on $\partial \mathcal A(S)$ defined using the Gromov product.
\end{rem}

Let $\mathcal T$ be the countable set of finite length edge-paths in $\mathcal A(S)$, up to the action of $\Mod$, equipped with the discrete topology. For an infinite unicorn path $P=(c_i)_{i=0}^\infty$, given $i\geq 0$ and $j=i+1,\ldots,$ the subsurfaces $\Sigma_{i,j}\subseteq S$ filled by $c_i$ and $c_j$ form an ascending sequence $\Sigma_{i,i+1}\subseteq \Sigma_{i,i+2}\subseteq \cdots$ that stabilises with some subsurface which we call $\Sigma_i\subseteq S$. For each $i\geq 0$, let $m(i)>i+1$ be minimal satisfying $\Sigma_{i+1,m(i)}=\Sigma_{i+1}$.
Let $T_i=(c_j)_{j=i}^{m(i)}$, and let $[T_i]$ be the equivalence class of $T_i$ in $\mathcal T$.
Let $g\colon \mathcal P\to \mathcal T^\N$ be the map defined by $g(P)=([T_i])_{i=0}^\infty$. Let $E_t$ be the tail equivalence relation on $\mathcal T^\N$ described in Section~\ref{sec:intro} (with $\Omega=\mathcal T$).

Note that the definition of $g$ can be analogously extended to infinite unicorn paths $P\notin \mathcal P$ (i.e.\ to infinite unicorn paths that start at points distinct from $a_0$), which we will make use of later on.

\begin{rem}
\label{rem:Borel2}
We equip $\mathcal T^\N$ with the product topology. Then the map $g\colon \mathcal P\to \mathcal T^\N$ is Borel. Indeed, for all $0\leq i<j,$ the maps $P\to \Sigma_{i,j}$ are continuous maps from $\mathcal P$ to the countable discrete set of subsurfaces of~$S$, and hence their limits $P\to \Sigma_{i}$ are Borel. Thus, for all $0\leq i<j$, the subset of $\mathcal P$ defined by the identity $\Sigma_{i,j}=\Sigma_i$ is Borel, and so the maps $m(i)\colon \mathcal P\to \N$ are Borel. Consequently, all the maps $ [T_i]\colon \mathcal P\to \mathcal T$ are Borel, as desired.
\end{rem}

\begin{lemma}
\label{lem:coding}
Let $P,P'\in \mathcal P$. If $g(P)\sim_{E_t}g(P')$, then there is $\psi\in \Mod$ satisfying $\psi f(P)=f(P')$. Conversely, for each orbit $\omega$ of the action of $\Mod$ on $\partial \mathcal A(S)$, there are finitely many equivalence classes of $E_t$ on $\mathcal T^\N$ containing all $g(P)$ for $P\in \mathcal P$ with $f(P)\in \omega$.
\end{lemma}
\begin{proof}
Denote $P=(c_i)_{i=0}^\infty,P'=(c'_i)_{i=0}^\infty$. Let $T_i'$ be defined for $P'$ analogously as $T_i$ was for $P$. If $g(P)\sim_{E_t}g(P')$, then there are $k\in \Z,j\in \N$ such that
$[T_i]=[T'_{i+k}]$ for all $i\geq j$. In particular, there is $\psi\in \Mod$ with $\psi T_j=T'_{j+k}$. We will show inductively that $\psi T_i=T'_{i+k}$ for all $i\geq j$, so in particular $\psi c_i=c'_{i+k}$ implying $\psi f(P)=f(P')$.

Suppose that we have established $\psi T_i=T'_{i+k}$ for some $i\geq j$. If $m(i+1)\leq m(i)$, then $\psi T_{i+1}=T'_{i+1+k}$ is immediate, so we can assume $m(i+1)> m(i)$. Let $\rho\in \Mod$ be such that $\rho T_{i+1}=T'_{i+1+k}$. Then $\rho^{-1}\psi$ fixes all $c_{i+1},\ldots ,c_{m(i)}$. Thus the restriction of $\rho^{-1}\psi$ to the subsurface $\Sigma_{i+1}\subset S$, which $c_{i+1}$ and $c_{m(i)}$ fill,
is the identity map. By the definition of $\Sigma_{i+1}$, we have that $c_{m(i)+1}, \ldots, c_{m(i+1)}$ all lie in~$\Sigma_{i+1}$. This implies that $\rho^{-1}\psi$ fixes them and so $\psi T_{i+1}=T'_{i+1+k}$, completing the induction.

For the converse, let $\omega$ be the orbit under $\Mod$ of some $F(L_0)\in \partial \mathcal A(S)$. Let $l_1,
\ldots, l_n$ be the finitely many geodesic lines asymptotic to~$L_0$. Choose $a_1,\ldots, a_p\in A$ that are representatives of the equivalence classes of $\sim_{L_0}$ distinct from the one containing~$a_0$. For $0\leq q\leq p$ and $1\leq j\leq n$, let $P_{qj}=P(a_q, l_j)$.

Let $P=(c_i)_{i=0}^\infty\in \mathcal P$ with $f(P)\in \omega$. By Theorem~\ref{thm:Pho-On}, we have $P=P(a_0,l)$ where $l$ is a geodesic line asymptotic to $\psi L_0$ for some $\psi\in \Mod$. In particular, for some $1\leq j\leq n$ we have $l=\psi l_j$. Thus $\psi^{-1}P=P(\psi^{-1}a_0, l_j)$. Choose $0\leq q\leq p$ so that $a_q\sim_{L_0}\psi^{-1}a_0$. By Lemma~\ref{lem:unite}, writing ${P}_{qj}=(b_i)_{i=0}^\infty$, we have $\psi^{-1}c_i=b_{i+k}$ for some $k\in \Z$ and all $i$ sufficiently large. We have then that $g({P}_{qj})$ and $g(\psi^{-1}P)$, hence also $g(P)$, are tail equivalent.
\end{proof}

\begin{proof}[Proof of Theorem~\ref{thm:arc}]
Write $E$ for the equivalence relation on
$\partial\mathcal{A}(S)$ induced by the action of
$\mathrm{Mod}(S)$ and write $E^*$ for the equivalence relation
on $\mathcal P$ that is the pullback of
$E$ via $f$, i.e.\ $P\sim_{E^*} P'$ if $f(P)\sim_E
f(P')$. Since $E$ is Borel and countable, and $f$ is Borel and finite-to-one, we have that $E^*$ is also Borel and countable.

Since $f$ is a Borel finite-to-one function, it
has a Borel right inverse by the Lusin--Novikov
uniformisation theorem \cite[Thm~18.10]{Kech}. Consequently, $E$ is Borel reducible to $E^*$.
Thus it is enough to show that $E^*$ is hyperfinite.

Write $E_t^*$ for the equivalence relation on $\mathcal P$ that is the pullback of~$E_t$ via~$g$. Since $E_t$ is Borel, and $g$ is Borel, we have that $E_t^*$ is Borel. By Lemma~\ref{lem:coding}, we have $E_t^*\subseteq E^*$ and every equivalence class of $E^*$
contains finitely many equivalence classes of $E^*_t$. (In particular, $E_t^*$ is countable.)
Thus by \cite[Prop~1.3(vii)]{jkl} it is enough to show that $E_t^*$ is hyperfinite.

Note that $g$ is a Borel reduction of $E_t^*$ to $E_t$. Thus since $E_t$ is
hyperfinite \cite[Cor~8.2]{djk}, we have that $E_t^*$ is
hyperfinite as well.
\end{proof}

\begin{proof}[Proof of Corollary~\ref{cor:curve}]
Assume first that $S$ has $n\geq 1$ punctures. Then by \cite[Thm~1.3]{Kla}, Theorem~\ref{thm:arc}, and \cite[Prop~1.3(iii)]{jkl} it suffices to prove that $\mathcal {EL}(S)$ is a Borel subset of $\mathcal {EL}_0(S)$.
Indeed, $L_0\in \mathcal {EL}_0(S)$ is a minimal filling lamination if and only if each geodesic representative of a curve $c$ on $S$ intersects~$L_0$ and does it transversally. Given $c$, this is an open condition, and so $\mathcal {EL}(S)$ is a countable intersection of open sets.

Secondly, assume $n=0$ and let $S'$ be the surface obtained from~$S$ by adding one puncture at a point outside the closure of the union of all embedded geodesic circles and
lines, which exists by \cite[Thm~I]{BS}. This
induces a closed embedding
$e\colon \mathcal {EL}(S)\to \mathcal {EL}(S')$, which is a section for the map $r\colon \mathcal {EL}(S')\to \mathcal {EL}(S)$ defined by forgetting the puncture. See \cite[\S4.2]{Pho} for details.
Thus for each $L_1,L_2\in \mathcal {EL}(S)$, with $\psi'e(L_1)=e(L_2)$ for some $\psi'\in \mathrm{Mod}(S')$, we have that the image $\psi\in \mathrm{Mod}(S)$ of $\psi'$ under the puncture forgetting map
$\mathrm{Mod}(S')\to\mathrm{Mod}(S)$ satisfies $\psi(L_1)=L_2$.

Conversely, let $L\in \mathcal {EL}(S)$ and let $R_1,\ldots, R_n\subset S$ be the components of $S\setminus L$. For $1\leq j\leq n$, let $L_j$ be a lamination in $\mathcal {EL}(S')$ obtained from~$L$ by adding a puncture in $R_j$, under an arbitrary identification with~$S'$. All such identifications differ by $\mathrm{Mod}(S')$, so the resulting orbit $[L_j]$ in $\mathcal {EL}(S')$ does not depend on our choice. Since $e$ is a section for $r$, we have $e(L)\in \bigcup_{j=1}^n [L_j]$. Analogously, for any
$\psi\in \Mod$, we have $e(\psi(L))\in \bigcup_{j=1}^n [L_j]$.

Consequently, under the identification of $\mathcal {EL}(S)$ with $e(\mathcal {EL}(S))$, each orbit of $\Mod$ on $\mathcal {EL}(S)$ consists of the intersections of finitely many orbits of $\mathrm{Mod}(S')$ on $\mathcal {EL}(S')$ with $e(\mathcal {EL}(S))$. Thus by \cite[Prop~1.3 (iii,vii)]{jkl}, the hyperfiniteness of the action of $\mathrm{Mod}(S)$ on $\mathcal {EL}(S)$ follows from
the hyperfiniteness of the action of $\mathrm{Mod}(S')$ on $\mathcal {EL}(S')$.
\end{proof}

\section{Complete geodesic laminations}
\label{sec:complete}
A geodesic lamination $L$ on $S$ is \emph{complete}, if each component of $S\setminus L$ is an ideal triangle or a once-punctured monogon, and $L$ lies in the closure (in the Hausdorff topology on the space of compact subsets of $S$) of the set of embedded geodesic circles. By $\mathcal CL(S)$ we denote the space of complete geodesic laminations with the Hausdorff topology, which is compact and Hausdorff (see \cite[\S2.1]{Ham}). 

\begin{proof}[Proof of Corollary~\ref{cor:CL}]
For a complete geodesic lamination $L$, let $L'$ denote the union of minimal sublaminations of $L$ that are not embedded geodesic circles. Let 
$Y(L)$ denote the subsurface of $S$ filled by $L'$. Given a subsurface $Y\subseteq S$, let $\mathcal CL(S,Y)\subset \mathcal CL(S)$ denote the subspace of laminations $L$ with $Y(L)=Y$. We claim that each $\mathcal CL(S,Y)$ is a Borel subset of $\mathcal CL(S)$, and hence $\Mod  \mathcal CL(S,Y)$ is Borel as well. 

Indeed, for a lamination $L$ in $\mathcal CL(S,Y)$, the union $Z$ of the non-isolated leaves of $L$ on $S\setminus Y$ is a union of disjoint (geodesic representatives of) curves on $S\setminus Y$. Thus a complete geodesic lamination $L$ belongs to $\mathcal CL(S,Y)$ if and only if 
\begin{itemize}
\item
each curve $c$ on $Y$ intersects $L$ transversally infinite number of times ($G_\delta$ condition), and 
\item
there exists a union of disjoint curves $Z$ on $S\setminus Y$, such that each curve $c$ on $S\setminus (Y\cup Z)$ or in $Z\cup\partial Y$	does not intersect $L$ transversally infinite number of times ($F_{\sigma\delta\sigma}$ condition). 
\end{itemize}
Thus $\mathcal CL(S,Y)$ is an $F_{\sigma\delta\sigma}$ set, justifying the claim.

By \cite[Prop~1.3(v)]{jkl}, to prove that the orbit equivalence relation on  $\mathcal CL(S)$ induced by the action on $\Mod$ is hyperfinite, it suffices to show that its restriction to each $\Mod  \mathcal CL(S,Y)$ is hyperfinite. By \cite[Prop~1.3(vi)]{jkl}, it suffices to show that the orbit equivalence relation $E$ on  $\mathcal CL(S,Y)$ induced by the action of the stabiliser $\Mod_Y$ of $Y$ in $\Mod$ is hyperfinite.

Let $Y_1,\ldots, Y_k$ be the components of $Y$, where we treat all geodesic boundary components as punctures. Let
$g\colon \mathcal CL(S,Y)\to \mathcal{EL}(Y_1)\times \cdots \times \mathcal{EL}(Y_k)$ be the map assigning to each $L$ the components of its sublamination $L'$. In the case where $Y=\emptyset$, the product $\mathcal{EL}(Y_1)\times \cdots \times \mathcal{EL}(Y_k)$ should be understood as a point. By Corollary~\ref{cor:curve} and \cite[Prop~1.3(iv)]{jkl}, the orbit equivalence relation on $\mathcal{EL}(Y_1)\times \cdots \times \mathcal{EL}(Y_k)$ induced by the action of $\mathrm{Mod}(Y_1)\times \cdots \times \mathrm{Mod}(Y_k)$ is hyperfinite. The group $\mathrm{Mod}(Y_1)\times \cdots \times \mathrm{Mod}(Y_k)$ is of finite index in $\mathrm{Mod}(Y_1\sqcup \cdots \sqcup Y_k)$. Thus by \cite[Prop~1.3(vii)]{jkl}, the orbit equivalence relation on $\mathcal{EL}(Y_1)\times \cdots \times \mathcal{EL}(Y_k)$ induced by the action of $\mathrm{Mod}(Y_1\sqcup \cdots \sqcup Y_k)$ is hyperfinite. Its pullback~$F$ under~$g$ is thus hyperfinite as well, since $g$ has countable fibers. Since $E$ is contained in $F$, it is hyperfinite by \cite[Prop~1.3(i)]{jkl}, as desired.
\end{proof}

\end{document}